\documentclass[11pt]{article}
\usepackage{amsmath}
\usepackage{amssymb}
\usepackage{amsthm}
\usepackage{geometry}
\usepackage{hyperref}
\usepackage{listings}
\usepackage{xcolor}
\usepackage{enumitem}
\usepackage[T1]{fontenc}
\usepackage{comment}
\usepackage{tikz,array}
\usepackage{indentfirst}
\usepackage{subcaption}

\definecolor{codegreen}{rgb}{0,0.6,0}
\definecolor{codegray}{rgb}{0.5,0.5,0.5}
\definecolor{codepurple}{rgb}{0.58,0,0.82}
\definecolor{backcolour}{rgb}{0.95,0.95,0.92}

\lstdefinestyle{cppstyle}{
    backgroundcolor=\color{backcolour},
    commentstyle=\color{codegreen},
    keywordstyle=\color{blue},
    numberstyle=\tiny\color{codegray},
    stringstyle=\color{codepurple},
    basicstyle=\ttfamily\footnotesize,
    breakatwhitespace=false,
    breaklines=true,
    captionpos=b,
    keepspaces=true,
    numbers=left,
    numbersep=5pt,
    showspaces=false,
    showstringspaces=false,
    showtabs=false,
    tabsize=4
}

\title{Distribution of independent sets in perfect $r$-ary trees}
\author{Daniel I\v{l}kovi\v{c}\thanks{Institute of Mathematics, Leipzig University, Augustusplatz 10, 04109 Leipzig, Germany. Mathematics Institute, University of Warwick, Coventry CV4 7AL, UK. \url{daniel.ilkovic@warwick.ac.uk}. Supported by the Alexander von Humboldt Foundation in the framework of the Alexander von Humboldt Professorship of Daniel Krá\v{l} endowed by the Federal Ministry of Education and Research.} \and Jun Yan\thanks{Mathematical Institute, University of Oxford, Oxford OX2 6GG, UK. Email: \url{jun.yan@maths.ox.ac.uk}. Supported by
ERC Advanced Grant 883810.}}

\newtheorem{theorem}{Theorem}[section]
\newtheorem{lemma}[theorem]{Lemma}
\newtheorem{conjecture}[theorem]{Conjecture}

\newtheorem{question}[theorem]{Question}

\theoremstyle{definition}
\newtheorem{definition}[theorem]{Definition}

\newcommand{\cI}{\mathcal{I}}
\DeclareMathOperator{\cas}{CAS}

\newcounter{propcounter}

\begin{document}
\maketitle

\begin{abstract}
Given a graph $G$, the family of all independent sets of size $k$ containing a fixed vertex $v$ is called a star with centre $v$, and is denoted by $\cI_G^k(v)$. Motivated by a generalisation of the Erd\H{o}s-Ko-Rado Theorem to the setting of independent sets in graphs, Hurlbert and Kamat conjectured that for every tree $T$ and every $k$, the maximum of $|\cI_T^k(v)|$ can always be attained by a leaf of $T$. While this conjecture turns out to be false in general, it is known to hold for specific families of trees like spiders and caterpillars. In this paper, we prove that this conjecture holds for a new family of trees, the perfect $r$-ary trees, by constructing injections from stars centred at arbitrary vertices to stars centred at leaves. We also show that the analogous property holds for every forest $\mathcal{T}$ that is the disjoint union of perfect trees with possibly varying sizes and arities, and determine the leaf that maximises $|\cI_{\mathcal{T}}^k(v)|$. 

\end{abstract}

\section{Introduction}
A family $\mathcal{F}$ of sets is called \textit{intersecting} if any $A,B\in\mathcal{F}$ satisfies $A\cap B\not=\varnothing$. The well-known Erd\H os-Ko-Rado Theorem~\cite{EKR}, published in 1961, states that if $n\geq2k$ and $\mathcal{F}$ is an intersecting family of subsets of $[n]$ of size $k$, then $|\mathcal{F}|\leq\binom{n-1}{k-1}$. Moreover, if $n>2k$, then equality holds if and only if $\mathcal{F}=\{A:i\in A, |A|=k\}$ for some $i\in[n]$. Since then, there have been many generalisations and variants of the Erd\H os-Ko-Rado Theorem in the literature. In this paper, we focus on a generalisation to the following setting of independent sets in graphs proposed by Holroyd, Spencer, and Talbot~\cite{HST,HT} in 2005. 

\begin{definition}
Let $G=(V(G), E(G))$ be a graph, let $v\in V(G)$, and let $k$ be an integer.  
\begin{itemize}
    \item A subset $I \subseteq V(G)$ is an \emph{independent set} in $G$ if $uw \notin E(G)$ for all $u, w \in I$. 
    \item The family of all independent sets of size $k$ in $G$ is denoted by $\mathcal{I}_{G}^{k}$.
    \item The family $\{A\in\mathcal{I}_{G}^{k}:v\in A\}$ of all independent sets of size $k$ in $G$ containing $v$ is denoted by $\mathcal{I}_{G}^{k}(v)$. Each family of the form $\cI_G^k(v)$ is called a \emph{star}, and $v$ is called its \emph{star centre}.
\end{itemize}
\end{definition}
Note that every star $\cI^k_G(v)$ is an intersecting family in $\cI^k_G$ by definition.

\begin{definition}
Let $G$ be a graph and let $k$ be an integer.
\begin{itemize}
    \item $G$ is \textit{$k$-EKR} if no intersecting family of sets in $\cI_G^k$ is larger than the largest star in $\cI_G^k$.
    \item $G$ is \textit{strictly $k$-EKR} if every intersecting family of sets in $\cI_G^k$ of maximum size is a star. 
\end{itemize}
\end{definition}
Observe that using this language, the Erd\H os-Ko-Rado Theorem can be restated as saying the empty graph on $n$ vertices is $k$-EKR if $n\geq 2k$, and strictly $k$-EKR if $n>2k$. 

Several families of graphs have been shown to be $k$-EKR for suitable values of $k$. These include powers of cycles~\cite{T}, powers of paths~\cite{HST}, and certain graphs with an isolated vertex~\cite{BH,HST,HK}. A major outstanding conjecture due to Holroyd and Talbot~\cite{HT} states that if $\mu(G)$ is the minimum size of a maximal independent set in $G$, then $G$ is $k$-EKR if $k\leq\mu(G)/2$, and strictly so if $k<\mu(G)/2$. Many positive results in this direction proceed by induction and use a shifting/compression proof technique involving the deletion of edges and vertices. The presence of an isolated vertex $x$ in $G$ is very useful as it is not affected by the deletion process, and due to the easy observation that for every $k$, the maximum of $|\cI_G^k(v)|$ can always be attained by $x$.



For a graph $G$ without an isolated vertex, finding the centre of a star of maximum size is no longer trivial, and it seems difficult in general to prove that $G$ is $k$-EKR without knowing which star has maximum size. Motivated by this, Hurlbert and Kamat~\cite{HK} proposed to study the intermediate problem of determining which vertex $v\in V(G)$ maximises $|\cI^k_G(v)|$. In particular, they considered the case when $G$ is a tree, and conjectured that there is always a maximum-sized star centred at a leaf. 
\begin{conjecture}[\cite{HK}]\label{conj:HK}
For any $k\ge1$ and any tree $T$, there exists a leaf $\ell$ of $T$ such that $|\mathcal{I}_{T}^{k}(v)|\le|\mathcal{I}_{T}^{k}(\ell)|$ for every $v\in V(T)$.
\end{conjecture}

Hurlbert and Kamat proved this conjecture for $k\le4$~\cite{HK}. However, independent work of Baber~\cite{Baber} and Borg~\cite{Borg} produced counterexamples disproving Conjecture~\ref{conj:HK} for every $k\geq5$. Feghali, Johnson, and Thomas~\cite{FJT} produced further counterexamples showing that maximum-sized stars can be centred on a vertex with arbitrarily large degree and arbitrarily far away from any leaves in $T$, which disprove Conjecture~\ref{conj:HK} in a very strong sense. 

Even though Conjecture~\ref{conj:HK} is now disproved, it remains interesting to determine which trees satisfy the condition in Conjecture~\ref{conj:HK}. We use the following terminology introduced by Estrugo and Pastine in~\cite{EP}.
\begin{definition}
Let $T$ be a tree and let $k\geq1$ be an integer. 
\begin{itemize}
    \item $T$ is \textit{$k$-HK} if there exists a leaf $\ell$ of $T$ such that $|\mathcal{I}_{T}^{k}(v)|\le|\mathcal{I}_{T}^{k}(\ell)|$ for each $v\in V(T)$.
    \item $T$ is \textit{HK} if it is $k$-HK for every $k\geq1$.
\end{itemize}
\end{definition}
Two families of trees known to be HK are spiders~\cite{HK2} and caterpillars~\cite{EP}. We refer interested readers to a detailed survey~\cite{H} by Hurlbert of other results on graphs with EKR-type and/or HK-type properties.



The main result of this paper, Theorem~\ref{thm:main}, shows that the perfect $r$-ary trees are HK. For every integer $r\geq2$, a \textit{perfect $r$-ary tree} with $h$ \textit{levels} is the unique rooted plane tree $T$ that has exactly $r^{h-1}$ leaves, all at distance $h-1$ away from the root, and whose non-leaf vertices all have exactly $r$ children. The \textit{arity} of a perfect $r$-ary tree is $r$. A \textit{perfect tree} is a perfect $r$-ary tree for some $r\geq2$. The $r=2$ case of Theorem~\ref{thm:main} answers a question of Hurlbert~\cite{H}. 

\begin{theorem}\label{thm:main}
Let $T$ be a perfect $r$-ary tree. For any integer $k \geq 1$, any vertex $v \in V(T)$, and any leaf $\ell \in V(T)$, we have $|\mathcal{I}_{T}^{k}(v)| \leq |\mathcal{I}_{T}^{k}(\ell)|$.
\end{theorem}

Using similar ideas, we also show that if $\mathcal{T}$ is a disjoint union of perfect trees, then the maximum star size can similarly be attained by a star centred at a leaf. Moreover, we determine which leaf achieves the maximum. 
\begin{theorem}\label{thm:forest}
Let $\mathcal{T}$ be a forest with components $T_1,T_2,\ldots,T_n$, such that $T_i$ is a perfect tree for every $i\in[n]$. Let $T_1,T_2,\ldots,T_m$ be the perfect trees with the highest arity among $T_1,T_2,\ldots,T_n$, let $h_i$ be the number of levels in $T_i$ for every $i\in[m]$, and suppose $h_1\leq h_2\leq\cdots\leq h_m$. 

If $h_i$ is even for every $i\in[m]$, then let $\ell$ be any leaf in $T_m$. Otherwise, let $j\in[m]$ be minimal subject to $h_j$ being odd, and let $\ell$ be any leaf in $T_j$. Then, for every positive integer $k$ and any $v\in V(\mathcal{T})$, $|\cI_{\mathcal{T}}^k(v)|\leq|\cI_{\mathcal{T}}^k(\ell)|$.
\end{theorem}

While we have shown that perfect $r$-ary trees are HK, it remains open whether they are $k$-EKR for suitable choices of $k$, so we pose the following question.
\begin{question}
For which values of $k,r,h$ is the perfect $r$-ary tree with $h$ levels $k$-EKR?
\end{question}
 
\section{Perfect $r$-ary trees}\label{sec:main}
Let $T$ be a perfect $r$-ary tree. For every non-leaf vertex $v\in V(T)$, denote its $r$ children from left to right by $c_1(v),c_2(v),\ldots,c_r(v)$.  For every non-root vertex $v\in V(T)$, denote its parent by $p(v)$. Due to the symmetries in the perfect $r$-ary tree $T$, the size of the star $|\mathcal{I}_{T}^{k}(v)|$ depends only on the level of $v$, so in particular $|\mathcal{I}_{T}^{k}(\ell)|$ is independent of the choice of a leaf $\ell\in V(T)$. Therefore, from now on we will assume without loss of generality that $\ell$ is the leftmost leaf of $T$ and $v$ is an ancestor of $\ell$. Denote the path from $v$ to $\ell$ by $P_{v,\ell}$.

To prove Theorem \ref{thm:main}, it suffices to find an injection from $\cI_{T}^k(v)$ to $\cI_T^k(\ell)$. We define the following families of independent sets of size $k$:
\begin{itemize}
  \item $\mathcal{A}_{v,\ell} = \{I \in \cI^k_T : v \in I,\ell\notin I\}$
  \item $\mathcal{B}_{v,\ell} = \{I \in \cI^k_T : v \notin I,\ell \in I\}$
  \item $\mathcal{C}_{v,\ell} = \{I \in \cI^k_T : v \in I,\ell \in I\}$
\end{itemize}
Since $|\mathcal{I}_{T}^{k}(v)| = |\mathcal{A}_{v,\ell}| + |\mathcal{C}_{v,\ell}|$ and $|\mathcal{I}_{T}^{k}(\ell)| = |\mathcal{B}_{v,\ell}| + |\mathcal{C}_{v,\ell}|$, Theorem~\ref{thm:main} is equivalent to showing $|\mathcal{A}_{v,\ell}| \leq |\mathcal{B}_{v,\ell}|$. Therefore, it suffices to construct an injection $\phi: \mathcal{A}_{v,\ell} \to \mathcal{B}_{v,\ell}$. 

For every independent set $I\in\mathcal{A}_{v,\ell}$ containing $v$ but not $\ell$, the corresponding independent set $\phi(I)\in\mathcal{B}_{v,\ell}$ containing $\ell$ but not $v$ will be constructed by an iterative process. The first step of this process is to remove $v$ from $I$ and replace it with $\ell$. Of course, this may create a conflict if the parent of $\ell$ was originally in $I$. In fact, the goal of each subsequent iteration of the process is to check for and resolve a conflict created by the vertex movement in a previous iteration. We will show that this process must terminate and output the desired independent set $\phi(I)$, with the resulting mapping $I\mapsto\phi(I)$ being injective. This will be done in detail in Section~\ref{sec:even} and Section~\ref{sec:odd} depending on the parity of the distance $d(v,\ell)$. The case when $d(v,\ell)$ is even is more involved. To improve readability, we isolate part of the iterative process used there and present it as the CAS operation in Section~\ref{sec:cas}.

\subsection{The Conditional Alternating Swap (CAS)}\label{sec:cas}
When $d(v,\ell)$ is even, the key part of the iterative process we use in Section~\ref{sec:even} to construct the injection $\phi$ involves successively swapping vertices on the path $P_{v,\ell}$. To resolve conflicts that may arise in the rest of the tree $T$, we use the following CAS operation. See Figure~\ref{fig:1} for an example application of CAS when $r=2$.  
\begin{definition}[Conditional Alternating Swap]\label{def:CAS}
Let $T_d$ and $T_u$ be two disjoint perfect $r$-ary trees with roots $d$ and $u$, respectively. Suppose also that $T_d$ has an odd number of levels, and $T_u$ has more levels than $T_d$.

Let $I\subseteq V(T_d)\cup V(T_u)$ be an independent set containing $u$ but not $d$. Define the \textit{Conditional Alternating Swap (CAS)} operation that outputs an independent set $\cas(I,T_d,T_u)\subseteq V(T_d)\cup V(T_u)$ of the same size containing $d$ but not $u$ as follows.

\medskip

\noindent\textbf{Initialisation:} Create two queues $Q_1^s$ and $Q_1^t$, containing just $u$ and just $d$, respectively. Let $I_1=I$ and $V_1=\{u,d\}$.

\medskip

\noindent\textbf{Iteration Step:}
At the beginning of the $i$-th iteration, we have some $I_i,V_i,Q_i^s,Q_i^t\subseteq V(T_d)\cup V(T_u)$ such that the following conditions hold. 
\stepcounter{propcounter}
\begin{enumerate}[label = \textbf{\Alph{propcounter}\arabic{enumi}}]
    \item\label{cas:1} $|Q_i^s|=|Q_i^t|$ and $|I_{i}|=|I|$.
    \item\label{cas:2} $V_i=\cup_{j=1}^i(Q_j^s\cup Q_j^t)$ is the set of all vertices that have been in either queue. $T_d[V_i]$ is a subtree of $T_d$ and $T_u[V_i]$ is a subtree of $T_u$. Moreover, every vertex in $Q_i^s\cup Q_i^t$ is a leaf in either $T_d[V_i]$ or $T_u[V_i]$, and has not been modified by this process yet.
    \item\label{cas:3} If $v\in I_i$ and $p(v)\in I_i$, then $v\in Q_i^s$.
    \item\label{cas:4} $I_i\cap(Q_i^t\cup p(Q_i^t))=\varnothing$.
    \item\label{cas:5} The distance between $d$ and every vertex in $Q_i^s\cap V(T_d)$ is odd, while the distance between $d$ and every vertex in $Q_i^t\cap V(T_d)$ is even.
    \item\label{cas:6} Suppose $s_i\in Q_i^s$ and $t_i\in Q_i^t$ are the oldest elements in the two queues, then one of them is in $T_d$ and the other is in $T_u$. If $s\in V(T_d)$ and $t\in V(T_u)$, then the distance between $s$ and $d$ is the same as the distance between $t$ and $u$, and vice versa. 
\end{enumerate} 
Note that~\ref{cas:1}--\ref{cas:6} hold after initialisation, so the first iteration can be started.
\begin{enumerate}
    \item \textbf{Dequeue:} If $Q_i^s=\varnothing$, terminate, and return $\cas(I,T_d,T_u)=I_i$, noting that~\ref{cas:3} implies that $\cas(I,T_d,T_u)$ is an independent set. Otherwise, by~\ref{cas:1}, $Q_i^t\not=\varnothing$ as well, let the oldest element in $Q_i^s$ and $Q_i^t$ be $s_i$ and $t_i$, respectively.
    \item \textbf{Conflict Check:} If $s_i\notin I_i$, let $I_{i+1}=I_i$, $V_{i+1}=V_i$, $Q_{i+1}^s=Q_i^s\setminus\{s_i\}$, and $Q_{i+1}^t=Q_i^t\setminus\{t_i\}$. Note that~\ref{cas:1}--\ref{cas:6} still hold with $i+1$ in place of $i$, so we can advance to the $(i+1)$-th iteration. Otherwise, $s_i\in I_i$, and the current iteration continues.
    \item \textbf{Swap Action:} Let $I_{i+1}=(I_i\setminus \{s_i\}) \cup \{t_i\}$, note that $t_i\not\in I_i$ by~\ref{cas:4}, so $|I_{i+1}|=|I_i|=|I|$ by~\ref{cas:1}. By~\ref{cas:6}, one of $s_i$ and $t_i$ is in $T_d$, and the other is in $T_u$. We have three cases.
    \begin{itemize}
        \item If $s_i\in V(T_d)$, then the distance between $s_i$ and $d$ is odd by~\ref{cas:5}. Since $T_d$ has an odd number of levels, $s_i$ is not a leaf in $T_d$. By~\ref{cas:6} and since $T_u$ has more levels than $T_d$, $t_i$ is also not a leaf in $T_u$. Moreover, using~\ref{cas:2} and that $s_i\in Q_i^s$, we have that for every $j\in[r]$, $c_j(s_i)\not\in V_i$. Hence, neither $s_i$ nor any of its children has been modified by this process at the start of the $i$-th iteration. Since $I$ is an independent set and $s_i\in I_i$, we have $s_i\in I$ and so none of its children is in $I$ or $I_i$. Remove $s_i$ from $Q_i^s$ and $t_i$ from $Q_i^t$. For every $j\in[r]$ in order, add $c_j(s_i)$ to $Q_i^t$ and $c_j(t_i)$ to $Q_i^s$. Let the resulting updated queues be $Q_{i+1}^s$ and $Q_{i+1}^t$. 
        \item If $t_i\in V(T_d)$ is not a leaf in $T_d$, then $s_i\in V(T_u)$ is also not a leaf in $T_u$ by~\ref{cas:6}. Similar to above, remove $s_i$ from $Q_i^s$ and $t_i$ from $Q_i^t$. For every $j\in[r]$ in order, add $c_j(s_i)$ to $Q_i^t$ and $c_j(t_i)$ to $Q_i^s$. Let the resulting updated queues be $Q_{i+1}^s$ and $Q_{i+1}^t$.
        \item If $t_i\in V(T_d)$ is a leaf in $T_d$, then let $Q_{i+1}^s=Q_i^s\setminus\{s_i\}$ and $Q_{i+1}^t=Q_i^t\setminus\{t_i\}$.
    \end{itemize} 
This completes the $i$-th iteration, with all of~\ref{cas:1}--\ref{cas:6} maintained, so we can advance to the $(i+1)$-th iteration. 
\end{enumerate}

\noindent\textbf{Termination:} We show that this process must terminate and output an independent set $\cas(I,T_d,T_u)$ containing $d$ but not $u$. Indeed, $T_d$ and $T_u$ contain only finitely many vertices. By~\ref{cas:2} and the iterating process, each vertex can only be added to the two queues once, and each iteration removes an existing vertex from each queue. Thus, $Q_i^s$ eventually becomes empty. At termination, the output $\cas(I,T_d,T_u)$ is indeed an independent set by~\ref{cas:3}. It contains $d$ but not $u$ from the 1st iteration.
\end{definition}

\begin{figure}[h]
    \centering
    \input{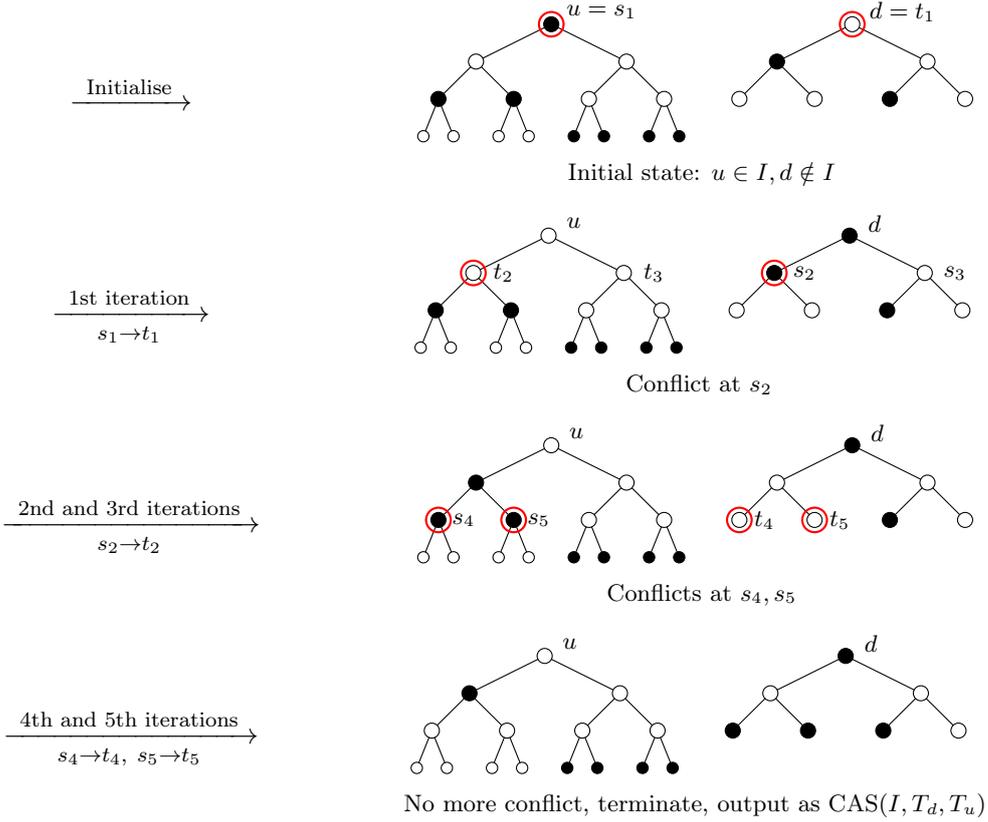}
    \caption{Example run of CAS with $r=2$. Vertices in $I_i$ are black, and vertices to be modified in each step are circled in red.}
    \label{fig:1}
\end{figure}

The following key lemma shows that the CAS operation is injective.
\begin{lemma}[Injectivity of CAS]\label{lem:injcas}
For $T_d$ and $T_u$ satisfying the assumptions in Definition~\ref{def:CAS}, the map $I\mapsto\cas(I,T_d,T_u)$ is injective. 
\end{lemma}
\begin{proof}
Let $I,I'\subseteq V(T_d)\cup V(T_u)$ be distinct independent sets containing $u$ but not $d$. We want to show that $\cas(I,T_d,T_u)\not=\cas(I',T_d,T_u)$. This is immediate if $|I|\not=|I'|$ as the CAS operation preserves size, so suppose $|I|=|I'|$. 

For each $i\geq 1$, let $I_i$ and $I_i'$ be the sets obtained at the beginning of the $i$-th iteration of the CAS process applied to $I$ and $I'$, respectively. 

Observe that if the CAS process applied to $I$ and the CAS process applied to $I'$ agree in their first $i-1$ iterations, then at the beginning of the $i$-th iteration, both processes have the same queues $Q_i^s$, $Q_i^t$, and so the same $s_i$ and $t_i$. If $Q_i^s$ is empty, then they both terminate, and we deal with this case below. If not, then the outcomes of the $i$-th iterations for $I$ and $I'$ entirely depend on whether $s_i\in I_i$ and whether $s_i\in I_i'$. By~\ref{cas:2}, $s_i$ has not been modified, so this is equivalent to whether $s_i\in I$ and whether $s_i\in I'$. If they agree, then the $i$-th iteration also agree, while if not, then as $s_i$ and $t_i$ are not modified in future iterations, the final outputs $\cas(I,T_d,T_u)$ and $\cas(I',T_d,T_u)$ differ as one of them contains $t_i$ and the other one doesn't. Therefore, the outputs differ if they disagree in some iteration of the process.

If both processes agree in each iteration up until termination, then they must terminate at the same time, say at the beginning of the $i$-th iteration. From the paragraph above, $I$ and $I'$ agree on all vertices that had been in $Q_j^s$ for some $j\in[i]$. They also agree on all vertices that have been in $Q_j^t$ for some $j\in[i]$, as by~\ref{cas:2} and~\ref{cas:4}, these vertices are not in $I$ and not in $I'$. Hence, $I$ and $I'$ agree on $V_i$, the set of vertices that have been in either queue. However, $I$ and $I'$ are distinct, so they disagree on some $x\in(V(T_d)\cup V(T_u))\setminus V_i$. Since only vertices in $V_i$ are modified by this process, $\cas(I,T_d,T_u)\not=\cas(I',T_d,T_u)$ as they also disagree on $x$.
\end{proof}

\subsection{Case 1: $d(v,\ell)$ is even}\label{sec:even}
When $d(v, \ell)$ is even, the iterative process defining the injection $\phi_{\textup{even}}:\mathcal{A}_{v,\ell}\to\mathcal{B}_{v,\ell}$ 
proceeds by a sequence of vertex swaps on the path $P_{v,\ell}$, combined with performing the CAS operation on the relevant adjacent subtrees. More precisely, let $I\in\mathcal{A}_{v,\ell}$, we construct $\phi_{\textup{even}}(I)\in\mathcal{B}_{v,\ell}$ of the same size as follows. 

\medskip

\noindent\textbf{Initialisation:} Set $I_1=(I\setminus\{v\})\cup\{\ell\}$, $d_1=p(\ell)$, and $u_1=c_1(v)$. 

\medskip

\noindent\textbf{Iteration Step:} For every integer $i\geq 1$, we say we are at the beginning of the $i$-th iteration if there exists $I_i$ and $d_1,d_2,\ldots,d_i,u_1,u_2,\ldots,u_i\in P_{v,\ell}$ such that the following conditions hold. In what follows, for every $v\in V(T)$, let $T_v$ be the subtree of $T$ induced by $v$ and all of its descendants in $T$. 
\stepcounter{propcounter}
\begin{enumerate}[label = \textbf{\Alph{propcounter}\arabic{enumi}}]
    \item\label{even:1} For every $j\in[i-1]$, $d_{j+1}=p(p(d_j))$ and $u_{j+1}=c_1(c_1(u_j))$. Moreover, either $d_i=u_i$ or $d_i$ is a descendant of $u_i$. 
    \item\label{even:2} Neither $u_i$ nor $p(u_i)$ is in $I_i$. 
    \item\label{even:3} Only vertices in $\{d_1,d_2,\ldots,d_{i-1},u_1,u_2,\ldots,u_{i-1}\}\cup(\cup_{j=1}^{i-1}\cup_{m=2}^r(V(T_{c_m(u_j)})\cup V(T_{c_m(d_j)})))$ have been modified, and every vertex is modified at most once.
    \item\label{even:4} $I_i$ is an independent set if $d_i\not\in I_i$. If $d_i\in I_i$, then the only edge in $I_i$ is $d_ic_1(d_i)$.
\end{enumerate}
Note that after initialisation,~\ref{even:1}--\ref{even:4} hold for $i=1$, so we can begin the first iteration. 

\begin{enumerate}
    \item\textbf{Primary Check:} If $d_i\not\in I_i$, terminate, and let $\phi_{\textup{even}}(I)=I_i$. Note that $\phi_{\textup{even}}(I)=I_i$ is an independent set by~\ref{even:4}. Otherwise, $d_i\in I_i$, continue. Observe that if $d_i\in I_i$, then by~\ref{even:2}, $d_i\not=u_i$. Also, by~\ref{even:1} and~\ref{even:3}, $d_i\in I_i$ implies that for every $2\leq m\leq r$, $c_m(d_i)$ has not been modified, and is in neither $I$ nor $I_i$.
    \item\textbf{Primary Shift:} Let $I_i'=(I_i\setminus\{d_i\})\cup\{u_i\}$. By~\ref{even:2} and~\ref{even:4}, the only possible edges in $I_i'$ are between $u_i$ and its $r$ children. 
    \item\textbf{Conditional Alternating Swap:} For every $2\leq m\leq r$, if $c_m(u_i)\not\in I_i'$, do nothing. Otherwise, we need to resolve the conflict at $c_m(u_i)$. Let $J_{i,m}=I_i'\cap(V(T_{c_m(d_i)})\cup V(T_{c_m(u_i)}))$. Using $d(v,\ell)$ is even,~\ref{even:1},~\ref{even:3}, and $c_m(d_i)\not\in I_i'$, the assumptions in Definition~\ref{def:CAS} are satisfied, so we can let $J_{i,m}'=\cas(J_{i,m},T_{c_m(d_i)},T_{c_m(u_i)})$, which is an independent set of the same size as $J_{i,m}$ that contains $c_m(d_i)$ but not $c_m(u_i)$. Let $I_i''=(I_i'\setminus(\cup_{m=2}^rJ_{i,m}))\cup(\cup_{m=2}^rJ_{i,m}')$. Then $|I_i''|=|I_i'|=|I|$, and the only possible edge in $I_i''$ is $u_ic_1(u_i)$.
    \item\textbf{Secondary Check:} If $c_1(u_i)\not\in I_i''$, then terminate, and let $\phi_{\textup{even}}(I)=I_i''$, which is an independent set. Otherwise, $c_1(u_i)\in I_i''$, continue, noting that $c_1(u_i)\not=p(d_i)$ as $d_i\in I_i$ implies that $p(d_i)\not\in I_i''$ by~\ref{even:3}.
    \item\textbf{Secondary Shift:} Let $I_{i+1}=(I_i''\setminus\{c_1(u_i)\})\cup\{p(d_i)\}$. Let $d_{i+1}=p(p(d_i))$ and $u_{i+1}=c_1(c_1(u_i))$, noting that $d_{i+1}$ is either equal to or a descendant of $u_{i+1}$ as $c_1(u_i)\not=p(d_i)$. Since $c_1(u_i)\in I_i''$, $u_{i+1}=c_1(c_1(u_i))$ is not in $I_{i+1}$ by~\ref{even:3}. It follows that~\ref{even:1}--\ref{even:4} hold with $i+1$ in place of $i$, so we can advance to the $(i+1)$-th iteration.
\end{enumerate}

\noindent\textbf{Termination:} Note that each iteration, if not terminated, reduces the distance between $u_i$ and $d_i$ by 4, so the process must terminate. From above, the output $\phi_{\textup{even}}(I)$ is in $\mathcal{B}_{v,\ell}$. 

\medskip

Similar to Lemma~\ref{lem:injcas}, we now show that $\phi_{\textup{even}}$ is injective.
\begin{lemma}\label{lem:injeven}
If $d(v,\ell)$ is even, then $\phi_{\textup{even}}:\mathcal{A}_{v,\ell}\to\mathcal{B}_{v,\ell}$ is injective. 
\end{lemma}
\begin{proof}
Let $I,\overline{I}\in\mathcal{A}_{v,\ell}$ be distinct, we need to show that $\phi_{\textup{even}}(I)\not=\phi_{\textup{even}}(\overline{I})$. This is immediate if $|I|\not=|\overline{I}|$ as $\phi_{\textup{even}}$ preserves size, so suppose $|I|=|\overline{I}|$.

If $I$ and $\overline{I}$ agree in every iteration of the process above, then the set of vertices that are modified in this process are the same for $I$ and $\overline{I}$. Let this set be $V$, then $I$ and $\overline{I}$ agree on $V$. Since $I\not=\overline{I}$, there exists $x\not\in V$ such that one of $I$ and $\overline{I}$ contains $x$, and the other does not. But since $x\not\in V$, $x$ is not modified by the process, so $\phi_{\textup{even}}(I)\not=\phi_{\textup{even}}(\overline{I})$ as they disagree on $x$.

Otherwise, let $i$ be minimal such that the above process for $I$ and $\overline{I}$ first disagree in the $i$-th iteration. Let $I_i$ and $\overline{I_i}$ be the sets in the beginning of the $i$-th iteration of the process for $I$ and $\overline{I}$, respectively. 

If $I$ and $\overline{I}$ disagree on $d_i$, then so do $I_i$ and $\overline{I_i}$ by~\ref{even:3}. From the Primary Check and Primary Shift steps and using~\ref{even:3}, $\phi_{\textup{even}}(I)\not=\phi_{\textup{even}}(\overline{I})$ as exactly one of them contains $u_i$. If $I$ and $\overline{I}$, and thus $I_i$ and $\overline{I_i}$ agree on $d_i$, then the Primary Check and Primary Shift steps are identical.

If $I_i'$ and $\overline{I_i}'$ disagree on $c_m(u_i)$ for some $2\leq m\leq r$, then the CAS step results in a disagreement between $\phi_{\textup{even}}(I)$ and $\phi_{\textup{even}}(\overline{I})$ on $c_m(d_i)$. Now suppose $I_i'$ and $\overline{I_i}'$ agree on $c_m(u_i)$ for all $2\leq m\leq r$. If $J_{i,m}$ and $\overline{J_{i,m}}$ differ for any $2\leq m\leq r$, then~\ref{even:3} and the injectivity of the CAS operation from Lemma~\ref{lem:injcas} imply that $\phi_{\textup{even}}(I)$ and $\phi_{\textup{even}}(\overline{I})$ differ. Otherwise, the CAS step is identical as well. 

Finally, if $I_i''$ and $\overline{I_i}''$ agree on $c_1(u_i)$, then the Secondary Check and Secondary Shift steps are identical, a contradiction to $I$ and $\overline I$ differing in the $i$-th iteration. If they disagree on $c_1(u_i)$ instead, then $\phi_{\textup{even}}(I)\not=\phi_{\textup{even}}(\overline I)$ as exactly one of them contains $p(d_i)$, completing the proof.
\end{proof}

\subsection{Case 2: Odd Distance $d(v,\ell)$}\label{sec:odd}
Now suppose $d(v,\ell)$ is odd. Let $I \in \mathcal{A}_{v,\ell}$, we construct $\phi_{\textup{odd}}(I)\in\mathcal{B}_{v,\ell}$ as follows. 

\medskip

\noindent\textbf{Initialisation:} Let $I_1=(I\setminus\{v\})\cup\{\ell\}$. Create two queues $Q_1^s$ and $Q_1^t$, containing just $p(\ell)$ and just $c_2(v)$, respectively. Let $L_1=\{\ell,p(\ell)\}$ and $V_1=\{v,c_2(v)\}$.

\medskip

\noindent\textbf{Iteration Step:}
At the beginning of the $i$-th iteration, we have some $I_i,L_i,V_i\subseteq V(T)$ such that the following conditions hold. 
\stepcounter{propcounter}
\begin{enumerate}[label = \textbf{\Alph{propcounter}\arabic{enumi}}]
    \item\label{odd:1} $|Q_i^s|=|Q_i^t|$ and $|I_{i}|=|I|$.
    \item\label{odd:2} $L_i\cup V_i=(\cup_{j=1}^i(Q_j^s\cup Q_j^t))\cup\{v,\ell\}$ is the set of all vertices in $T$ that have been in either queue, along with $v$ and $\ell$. $T[L_i]$ and $T[V_i]$ are disjoint subtrees of $T$. Every vertex in $Q_i^s\cup Q_i^t$ is a leaf in either $T[L_i]$ or $T[V_i]$, and has not been modified by this process.
    \item\label{odd:3} Every edge between vertices in $I_i$ contains a vertex in $Q_i^s$. For every $x\in Q_i^s$, it has exactly one neighbour in $I_i$, which is either $c_1(x)$ or $p(x)$, and it is $c_1(x)$ if and only if $x$ is in $L_i$ and is on $P_{v,\ell}$.
    \item\label{odd:4} $Q_i^t\cap I_i=\varnothing$. 
    \item\label{odd:5} The distance between $\ell$ and every vertex in $Q_i^s$ is odd, while the distance between $\ell$ and every vertex in $Q_i^t$ is even. 
    \item\label{odd:6} Suppose $s_i$ and $t_i$ are the oldest element in $Q_i^s$ and $Q_i^t$, respectively, then one of them is in $V_i$ and the other is in $L_i$.
\end{enumerate} 
Note that~\ref{odd:1}--\ref{odd:6} hold after initialisation, so the first iteration can be started.
\begin{enumerate}
    \item \textbf{Dequeue:} If $Q_i^s=\varnothing$, terminate, and return $\phi_{\textup{odd}}(I)=I_i$, noting that~\ref{odd:3} implies that $I_i$ is an independent set. Otherwise, by~\ref{odd:1}, $Q_i^t\not=\varnothing$ as well. Let the oldest element in $Q_i^s$ and $Q_i^t$ be $s_i$ and $t_i$, respectively. 
    \item \textbf{Conflict Check:} If $s_i\notin I_i$, let $I_{i+1}=I_i$, $L_{i+1}=L_i$, $V_{i+1}=V_i$, $Q_{i+1}^s=Q_i^s\setminus\{s_i\}$, $Q_{i+1}^t=Q_i^t\setminus\{t_i\}$, and advance to the $(i+1)$-th iteration, noting that~\ref{odd:1}--\ref{odd:6} hold for $i+1$. Otherwise, $s_i\in I_i$, continue with the current iteration.
    \item \textbf{Swap Action:} Let $I_{i+1}=(I_i\setminus \{s_i\}) \cup \{t_i\}$, noting that $t_i\not\in I_i$ by~\ref{odd:4}, so $|I_{i+1}|=|I_i|=|I|$ by~\ref{odd:1}. Since $s_i\in Q_i^s$, it has exactly one neighbour in $I_i$ by~\ref{odd:3}, which is either $c_1(s_i)$ or $p(s_i)$. By~\ref{odd:5}, $s_i$ is not a leaf in $T$, so it has $r$ children. Consider three cases.
    \begin{itemize}
        \item Suppose that the unique neighbour of $s_i$ in $I_i$ is $c_1(s_i)$, then $s_i\in L_i$ and $s_i$ is on $P_{v,\ell}$ by~\ref{odd:3}, and $t_i\in V_i$ by~\ref{odd:6}. Remove $s_i$ from $Q_i^s$ and $t_i$ from $Q_i^t$. If $t_i$ is not a leaf in $T$, then for every $2\leq m\leq r$ in order, add $c_m(t_i)$ to $Q_i^s$ and add $c_m(s_i)$ to $Q_i^t$. Then, note that by~\ref{odd:5} and using $d(v,\ell)$ is odd, we have that $p(s_i) \not= v$. Add $p(s_i)$ to $Q_i^t$ and add $c_1(t_i)$ to $Q_i^s$. If $t_i$ is a leaf in $T$, then do nothing. Let the resulting updated queues be $Q_{i+1}^s$ and $Q_{i+1}^t$.
        \item Suppose $s_i\in L_i$ and the unique neighbour of $s_i$ in $I_i$ is $p(s_i)$, then $t_i\in V_i$ by~\ref{odd:6}. Again, first remove $s_i$ from $Q_i^s$ and $t_i$ from $Q_i^t$. Then, if $t_i$ is not a leaf in $T$, for every $m\in[r]$ in order, add $c_m(s_i)$ to $Q_i^t$, and add $c_m(t_i)$ to $Q_i^s$, otherwise do nothing. Let the resulting updated queues be $Q_{i+1}^s$ and $Q_{i+1}^t$. 
        \item Suppose $s_i\in V_i$, then the unique neighbour of $s_i$ in $I_i$ is $p(s_i)$ by~\ref{odd:3}, and $t_i\in L_i$ by~\ref{odd:6}. Remove $s_i$ from $Q_i^s$ and $t_i$ from $Q_i^t$. If $t_i$ is a leaf in $T$, do nothing else. Otherwise, $t_i$ has $r$ children in $T$. For every $2\leq m\leq r$ in order, add $c_m(t_i)$ to $Q_i^s$ and add $c_m(s_i)$ to $Q_i^t$. Next, if $t_i$ is on $P_{v,\ell}$, then by~\ref{odd:2}, either $p(t_i)\not\in L_i\cup V_i$ or $p(t_i) = v$. If $p(t_i)\not=v$, add $p(t_i)$ to $Q_i^s$ and add $c_1(s_i)$ to $Q_i^t$, while if $p(t_i)=v$, do nothing. If instead $t_i$ is not on $P_{v,\ell}$, then $c_1(t_i)\not\in L_i\cup V_i$ by~\ref{odd:2}, in which case add it to $Q_i^s$ and add $c_1(s_i)$ to $Q_i^t$. Let the resulting updated queues be $Q_{i+1}^s$ and $Q_{i+1}^t$. 
    \end{itemize}
\end{enumerate}
Note that in all cases above,~\ref{odd:2} and~\ref{odd:3} imply that the new entrants to the queues are not in $L_i\cup V_i$. After updating $L_i$ to $L_{i+1}$ and $V_i$ to $V_{i+1}$ by including new elements in both queues appropriately, the $i$-th iteration is complete, with all of~\ref{odd:1}--\ref{odd:6} maintained, so we can advance to the $(i+1)$-th iteration. 

\medskip

\noindent\textbf{Termination:} The process must end as each vertex can only be added to either queue once, and each iteration removes a vertex from each queue, so $Q_i^s$ eventually becomes empty. The output $\phi_{\textup{odd}}(I)$ is in $\mathcal{B}_{v,\ell}$ because it is an independent set by~\ref{odd:3}, and it contains $\ell$ but not $v$ by the initialisation step. 

\medskip

Like before, it is straightforward to show that $\phi_{\textup{odd}}$ is injective. The proof is omitted as it is quite similar to that of Lemma~\ref{lem:injcas} and Lemma~\ref{lem:injeven}. 
\begin{lemma}\label{lem:injodd}
If $d(v,\ell)$ is odd, then $\phi_{\textup{odd}}:\mathcal{A}_{v,\ell}\to\mathcal{B}_{v,\ell}$ is injective. 
\end{lemma}

We can now prove Theorem~\ref{thm:main}. 
\begin{proof}[Proof of Theorem~\ref{thm:main}]
By Lemma~\ref{lem:injeven} and Lemma~\ref{lem:injodd}, $|\mathcal{A}_{v,\ell}|\leq|\mathcal{B}_{v,\ell}|$, so  \[|\mathcal{I}_{T}^{k}(v)| = |\mathcal{A}_{v,\ell}| + |\mathcal{C}_{v,\ell}|\leq|\mathcal{B}_{v,\ell}|+|\mathcal{C}_{v,\ell}|=|\cI_T^k(\ell)|.\qedhere\]
\end{proof}

\section{Multiple disjoint perfect trees}

In this section, we consider a forest $\mathcal{T}$ consisting of disjoint perfect trees. Note that every independent set in $\mathcal{T}$ is formed by taking the disjoint union of an independent set in each perfect tree within $\mathcal{T}$, and every such disjoint union of independent sets is an independent set in $\mathcal{T}$. Hence, by Theorem~\ref{thm:main}, for every $k\geq 1$, $|\cI_{\mathcal{T}}^k(v)|$ can be maximised by taking $v$ to be a leaf of some perfect tree in $\mathcal{T}$.  

To prove Theorem~\ref{thm:forest}, we need to determine which perfect tree in $\mathcal{T}$ should the maximising leaves come from. This is done via the following two lemmas. The first one compares the leaves of two perfect trees with different arities. 

\begin{lemma}\label{lemma:differentarity}
Let $\mathcal{T}$ be a forest consisting of disjoint perfect trees. Let $T_1, T_2$ be two perfect trees in $\mathcal{T}$ with arities $r_1$ and $r_2$, respectively, where $r_1 < r_2$. Let $\ell_1$ be a leaf of $T_1$ and $\ell_2$ be a leaf of $T_2$. Then $|\mathcal{I}_{\mathcal{T}}^{k}(\ell_1)| \le |\mathcal{I}_{\mathcal{T}}^{k}(\ell_2)|$.
\end{lemma}
\begin{proof}
By symmetry, we may assume that $\ell_1$ is the leftmost leaf of $T_1$, and $\ell_2$ is the leftmost leaf of $T_2$. Let $\mathcal{A}_{\ell_1,\ell_2}=\{I\in\cI^k_{\mathcal{T}}:\ell_1\in I, \ell_2\not\in I\}$ and $\mathcal{B}_{\ell_1,\ell_2}=\{I\in\cI^k_{\mathcal{T}}:\ell_1\not\in I, \ell_2\in I\}$. Like in Section~\ref{sec:main}, it suffices to construct an injection $\phi: \mathcal{A}_{\ell_1,\ell_2}\to\mathcal{B}_{\ell_1,\ell_2}$. Let $I \in \mathcal{A}_{\ell_1,\ell_2}$.
\begin{enumerate}
    \item\textbf{Leaf Swap:} Set $I_1= (I \setminus \{\ell_1\}) \cup \{\ell_2\}$.
    \item\textbf{Primary Check:} If $p(\ell_2)\not\in I_1$, then $I_1\in\mathcal{B}_{\ell_1,\ell_2}$. In this case, let $\phi(I)=I_1$ and terminate. Otherwise, $p(\ell_2)\in I_1$ and we have a conflict. Set $I_2=(I_1\setminus\{p(\ell_2)\})\cup\{p(\ell_1)\}$ and continue.
    \item\textbf{Secondary Check:} For each $2\leq i\leq r_1$, if $c_i(p(\ell_1))\in I_2$, then remove it from $I_2$ and replace it with $c_i(p(\ell_2))$, otherwise do nothing. If $p(p(\ell_1))\in I_2$, then remove it and replace it with $c_{r_1+1}(p(\ell_2))$, otherwise do nothing. Note that as $p(\ell_2)\in I_1$, $c_i(p(\ell_2))\not\in I_2$ for every $i\in[r_2]$, and $c_{r_1+1}(p(\ell_2))$ exists as $r_1<r_2$. Let $I_3$ be the resulting set. Observe that $I_3\in\mathcal{B}_{\ell_1,\ell_2}$. Let $\phi(I)=I_3$ and terminate.
\end{enumerate}
It is clear that $\phi$ is an injection, so $|\mathcal{I}_{\mathcal{T}}^{k}(\ell_1)| \le |\mathcal{I}_{\mathcal{T}}^{k}(\ell_2)|$.
\end{proof}

By Lemma~\ref{lemma:differentarity}, $|\cI_{\mathcal{T}}^k(v)|$ is maximised by the leaves of a perfect tree with the highest arity in $\mathcal{T}$. The next lemma compares leaves in two perfect trees with the same arity. 
\begin{lemma}\label{lemma:samearity}
Let $\mathcal{T}$ be a forest consisting of disjoint perfect trees. Let $T_1, T_2$ be perfect trees in $\mathcal{T}$ with equal arity $r$ containing $h_1$ and $h_2$ levels, respectively, where $h_1 < h_2$. Let $\ell_1 \in V(T_1)$ and $\ell_2 \in V(T_2)$ be leaves.
\begin{itemize}
    \item If $h_1$ is even, then $|\mathcal{I}_{\mathcal{T}}^{k}(\ell_1)| \le |\mathcal{I}_{\mathcal{T}}^{k}(\ell_2)|$.
    \item If $h_1$ is odd, then $|\mathcal{I}_{\mathcal{T}}^{k}(\ell_1)| \ge |\mathcal{I}_{\mathcal{T}}^{k}(\ell_2)|$.
\end{itemize}
\end{lemma}
\begin{proof}
Let $\mathcal{A}_{\ell_1,\ell_2}=\{I\in\cI^k_{\mathcal{T}}:\ell_1\in I, \ell_2\not\in I\}$ and $\mathcal{B}_{\ell_1,\ell_2}=\{I\in\cI^k_{\mathcal{T}}:\ell_1\not\in I, \ell_2\in I\}$. This is very similar to the proofs in Section~\ref{sec:main}, so we will be less formal and omit some details. 
Let $\varphi:V(T_1)\to V(T_2)$ be an injective homomorphism mapping $\ell_1$ to $\ell_2$. Define $\psi:V(T_1)\cup\varphi(V(T_1))\to V(T_1)\cup\varphi(V(T_1))$ by letting $\psi(v)=\varphi(v)$ if $v\in V(T_1)$, and $\psi(v)=\varphi^{-1}(v)$ if $v\in\varphi(V(T_1))$. In other words, $\psi$ maps every vertex in $V(T_1)\cup\varphi(V(T_1))$ to the corresponding vertex in its counterpart. 

\textbf{Case 1:} $h_1$ is even. We construct an injection $\phi:\mathcal{A}_{\ell_1,\ell_2}\to\mathcal{B}_{\ell_1,\ell_2}$ using the following algorithm. Initialise with $I_1 = (I \setminus \{\ell_1\}) \cup \{\ell_2\}$ and two queues $Q_1^s=\{p(\ell_2)\}$ and $Q_1^t=\{p(\ell_1)\}$. Throughout, we will maintain the key condition that if $v\in Q_i^s$, then $\psi(v)\in Q_i^t$ and $\psi(v)\not\in I_i$. 

In the $i$-th iteration, given $I_i$, $Q_i^s$, and $Q_i^t$, if $Q_i^s=\varnothing$, terminate and return $\phi(I)=I_i$. Otherwise, dequeue the oldest element $s_i$ in $Q_i^s$ and the corresponding element $\psi(s_i)$ in $Q_i^t$. If $s_i\not\in I_i$, advance to the next iteration with $I_{i+1}=I_i$. Otherwise, $s_i\in I_i$. Let $I_{i+1}=(I_i\setminus\{s_i\})\cup\{\psi(s_i)\}$. For each neighbour $v$ of $\psi(s_i)$ that has not been visited, add $v$ to $Q_i^s$ and add $\psi(v)$ to $Q_i^t$. Note that $\psi(v)$ also has not been visited, and since $\psi(v)$ is a neighbour of $s_i$ which is in $I_i$, we have $\psi(v)\not\in I_{i+1}$, so the key condition above is maintained. 

Furthermore, before termination, the vertices on the path $P$ between $p(\ell_1)$ and the root $x_1$ of $T_1$ alternate from being added to $Q_i^t$ and being added to $Q_i^s$, starting with $p(\ell_1)$ being added to $Q_i^t$ in the initialisation. Since $h_1$ is even, either the algorithm terminates before reaching $x_1$, or $x_1$ is added to $Q_i^t$, so vertices outside of $\varphi(V(T_1))$ in $T_2$ are never visited and the algorithm above is well-defined. Similar to the proofs in Section~\ref{sec:main}, the algorithm must terminate and output $\phi(I)\in\mathcal{B}_{\ell_1,\ell_2}$, and $\phi$ is injective. 

\textbf{Case 2:} $h_1$ is odd. We similarly construct an injection $\phi:\mathcal{B}_{\ell_1,\ell_2} \to \mathcal{A}_{\ell_1,\ell_2}$. Initialise with $I_1 = (I \setminus \{\ell_2\}) \cup \{\ell_1\}$ and two queues $Q_1^s=\{p(\ell_1)\}$ and $Q_1^t=\{p(\ell_2)\}$. Then, we proceed with the algorithm as in Case 1. Since $h_1$ is now odd and $p(\ell_1)$ is added to $Q_1^s$ in the initialisation, either the algorithm terminates before reaching the root $x_1$ of $T_1$, or $x_1$ is added to $Q_i^t$. This again shows that vertices outside of $\varphi(V(T_1))$ in $T_2$ are never visited and the algorithm is well-defined. At termination, it outputs $\phi(I)\in\mathcal{A}_{\ell_1,\ell_2}$, and $\phi$ is injective. 
\end{proof}

We can now put everything together and quickly deduce Theorem~\ref{thm:forest}.
\begin{proof}[Proof of Theorem~\ref{thm:forest}]
By Theorem~\ref{thm:main}, $|\cI_{\mathcal{T}}^k(v)|$ is maximised by the leaves of some perfect tree in $\mathcal{T}$. By Lemma~\ref{lemma:differentarity}, $|\cI_{\mathcal{T}}^k(v)|$ is maximised by a perfect tree with the highest arity in $\mathcal{T}$. Finally, by Lemma~\ref{lemma:samearity}, if every perfect tree with the highest arity in $\mathcal{T}$ has an even number of levels, then a maximising leaf belongs to the one with the highest number of levels, otherwise it belongs to the one with the lowest odd number of levels.
\end{proof}

\bibliographystyle{plain}  
\bibliography{reference}

@article{HK,
    AUTHOR = {Hurlbert, Glenn and Kamat, Vikram},
     TITLE = {Erd{\H o}s-{K}o-{R}ado theorems for chordal graphs and trees},
   JOURNAL = {Journal of Combinatorial Theory Series A},
    VOLUME = {118},
      YEAR = {2011},
    NUMBER = {3},
     PAGES = {829--841},
}

@phdthesis{Baber,
    author = {Rahil Baber},
    title = {Some {R}esults in {E}xtremal
{C}ombinatorics},
    school = {University College London},
    year = {2011}
}

@article {Borg,
    AUTHOR = {Borg, Peter},
     TITLE = {Stars on trees},
   JOURNAL = {Discrete Mathematics},
    VOLUME = {340},
      YEAR = {2017},
    NUMBER = {5},
     PAGES = {1046--1049},
}

@article {EP,
    AUTHOR = {Estrugo, Emiliano J.J. and Pastine, Adri\'an},
     TITLE = {On stars in caterpillars and lobsters},
   JOURNAL = {Discrete Applied Mathematics},
    VOLUME = {298},
      YEAR = {2021},
     PAGES = {50--55},
}

@article {FJT,
    AUTHOR = {Feghali, Carl and Johnson, Matthew and Thomas, Daniel},
     TITLE = {Erd{\H o}s-{K}o-{R}ado theorems for a family of trees},
   JOURNAL = {Discrete Applied Mathematics},
    VOLUME = {236},
      YEAR = {2018},
     PAGES = {464--471},
}

@article {EKR,
    AUTHOR = {Erd{\H o}s, Paul and Ko, Chao and Rado, Richard},
     TITLE = {Intersection theorems for systems of finite sets},
   JOURNAL = {The Quarterly Journal of Mathematics},
    VOLUME = {12},
    NUMBER = {1},
      YEAR = {1961},
     PAGES = {313--320},
}

@article {HT,
    AUTHOR = {Holroyd, Fred and Talbot, John},
     TITLE = {Graphs with the {E}rd{\H o}s-{K}o-{R}ado property},
   JOURNAL = {Discrete Mathematics},
    VOLUME = {293},
      YEAR = {2005},
    NUMBER = {1-3},
     PAGES = {165--176},
}

@article {HST,
    AUTHOR = {Holroyd, Fred and Spencer, Claire and Talbot, John},
     TITLE = {Compression and {E}rd{\H o}s-{K}o-{R}ado graphs},
   JOURNAL = {Discrete Mathematics},
    VOLUME = {293},
      YEAR = {2005},
    NUMBER = {1-3},
     PAGES = {155--164},
}

@article{H,
  title={A Survey of the {H}olroyd-{T}albot {C}onjecture},
  author={Hurlbert, Glenn},
  journal={arXiv:2501.16144},
  year={2025}
}

@article {HK2,
    AUTHOR = {Hurlbert, Glenn and Kamat, Vikram},
     TITLE = {On intersecting families of independent sets in trees},
   JOURNAL = {Discrete Applied Mathematics},
    VOLUME = {321},
      YEAR = {2022},
     PAGES = {4--9},
}

@article {T,
    AUTHOR = {Talbot, John},
     TITLE = {Intersecting families of separated sets},
   JOURNAL = {Journal of the London Mathematical Society},
    VOLUME = {68},
      YEAR = {2003},
    NUMBER = {1},
     PAGES = {37--51},
}

@article {BH,
    AUTHOR = {Borg, Peter and Holroyd, Fred},
     TITLE = {The {E}rd{\H o}s-{K}o-{R}ado properties of various graphs
              containing singletons},
   JOURNAL = {Discrete Mathematics},
    VOLUME = {309},
      YEAR = {2009},
    NUMBER = {9},
     PAGES = {2877--2885},
}

\end{document}